\documentclass[11pt,a4paper,reqno]{amsart}
\usepackage{amsmath, amssymb, amsthm, amscd, hyperref}
\usepackage{amsfonts}
\usepackage[utf8]{inputenc}
\usepackage[T1]{fontenc}
\usepackage[margin=1.1in]{geometry}
\newtheorem{theorem}{Theorem}[section]

\newtheorem{lemma}[theorem]{Lemma}
\newtheorem{proposition}[theorem]{Proposition}
\newtheorem{corollary}[theorem]{Corollary}
\newtheorem{remark}[theorem]{Remark}

\begin{document}
	\title[Patrick Omukuba]{The Moduli Space of Determinantal Representations of Cubic Surfaces and Invariant Theory of Root Systems}
	\author{Patrick Omukuba}
    \address{Institut für Mathematik ,JGU-Mainz}
    \email{paomukub@uni-mainz.de}

\date{\today}
	
	\begin{abstract}
		We present a framework for studying the moduli space of linear determinantal representations for flat families of complex projective cubic surfaces across  singularity boundaries. First, we deal with the classical deformation theory where a surface $S_0$ possesses a rational double point (RDP) of type $E_6$. By establishing a global simultaneous resolution over a finite ramified Galois covering of the global parameter slice, we realize the relative moduli space $\mathbb{H}_{\mathcal{N}}$ as a diagonal Weyl quotient $(\mathfrak{h} \times R)/W(E_6)$. Second, we extend this classification to the critical boundary configuration where $S_0$ contains a unique isolated simple elliptic singularity of type $\tilde{E}_6$. Because the local monodromy group becomes infinite, the classical simultaneous resolution framework completely breaks down. We bypass this obstruction by constructing a geometric substitute $\mathcal{Z}$  derived from the filtration of Looijenga's invariant algebra of affine Weyl groups. We also show that the vector bundle $\mathcal{Z}$  decomposes into a direct sum of line bundles over the structural elliptic curve $E$, with degrees explicitly matching the negative Coxeter marks of the highest root of $E_6$. Utilizing Riemann's extension theorem and vector bundle rigidity over elliptic curves, we settle our isomorphism result, proving that the moduli space $\bar{\mathbb{H}}$ representing the semiuniversal family  is globally isomorphic to the total space of $\mathcal{Z}$.
	\end{abstract}
	
	\maketitle
	
	\section{Introduction and Geometric Setup}
	
	A classical fundamental result in algebraic geometry dictates that every smooth projective cubic surface $S \subset \mathbb{P}^3$ can be realized as the blow-up of the projective plane $\mathbb{P}^2$ at six points in general linear position \cite{Dolgachev2012, Hartshorne1977}. Such a surface contains exactly 27 lines \cite{BruceWall1979}. Its underlying configuration geometry is intimately governed by the root system $R$ of the exceptional Lie algebra $E_6$. When these six points fall into special alignments, the anticanonical embedding maps the blown-up surface to a singular cubic surface $S_0 \subset \mathbb{P}^3$ containing isolated rational double points (RDPs), also known as canonical Du Val singularities \cite{Demazure1980,Slowdowy}. 
	
	Let $V$ be a four-dimensional complex vector space. Homogeneous cubic polynomials $\mathbb{P}(S^3 V^*) \cong \mathbb{P}^{19}$ parameterize projective cubic surfaces in $\mathbb{P}^3$. The classification of linear determinantal representations of such surfaces parameterizing the configurations of matrices whose determinants define the surface equation is deeply intertwined with the geometry of their lines and the linear systems of arithmetically Cohen--Macaulay (aCM) twisted cubics supported on their fibers \cite{Lehn2017, Buckley2006}. 
	
	When the surface $S$ is non-singular, it admits exactly 72 distinct equivalence classes of linear determinantal representations \cite{Buckley2006}. These correspond bijectively to the 72 root vectors of the root lattice of type $E_6$. When the normal surface $S_0$ is singular some lines collapse and hence the root system. This paper is two fold: First, we characterize the moduli space of determinantal representations of the semiuniversal family of the worst  isolated RDP, i.e of type $E_6$ in terms of the quotient by the finite Weyl group. Then, we extend to the infinite-monodromy situation of isolated Simple Elliptic Singularities of type $\tilde{E}_6$.
	
	\subsection*{Acknowledgments}The results in this paper are part of the authors PhD thesis \cite{Patrick Omukuba2026}, I am grateful to Manfred Lehn  for introducing me to the subject and his constant support. I also wish to express my sincere gratitude to the Institut für Mathematik-Johannes Gutenberg University Mainz for providing a conducive research environment. 
	
	\subsection{The RDP case}
	
	To eliminate coordinate redundancy arising from the natural action of $\text{PGL}(4,\mathbb{C})$, we fix a smooth, four-dimensional transversal slice $\mathcal{N} \subset \mathbb{P}^{19}$ to the $\text{GL}(V)$-orbit of $[S_0]$ passing through $[S_0]$. Restricting the universal family yields a flat global algebraic family of projective surfaces $p : \mathcal{S}_{\mathcal{N}} \to \mathcal{N}$. Brieskorn's classical construction models the local semiuniversal deformation space $\mathcal{U} \cong \mathbb{C}^6$ as the orbit space $\mathfrak{h}/W(E_6)$, where $\mathfrak{h}$ is the Cartan subalgebra \cite{Brieskorn1970}.	By constructing a global covering base space via the fiber product $\tilde{\mathcal{N}} := \mathcal{N} \times_\mathcal{U} \mathfrak{h}$, we unramify the finite local monodromy action of the Weyl group $W(E_6)$. This allows the point configurations to be tracked by smooth, non-intersecting sections, resolving the family into a smooth variety $\tilde{\mathcal{S}}$ via sequential algebraic blow-ups:
	
	\begin{theorem}\label{thm:rdp_resolution}
		Let $p : \mathcal{S}_{\mathcal{N}} \to \mathcal{N}$ be the semiuniversal deformation family of a projective cubic surface $S_0$ with an isolated rational double point of type $E_6$.
		\begin{enumerate}
			\item There exists a finite branched Galois covering $\psi : \tilde{\mathcal{N}} \to \mathcal{N}$ with Galois group isomorphic to the Weyl group $W(E_6)$, and a smooth variety $\tilde{\mathcal{S}}$ forming a flat, smooth projective fibration $\tilde{p} : \tilde{\mathcal{S}} \to \tilde{\mathcal{N}}$.
			\item There exists a proper birational morphism {$\pi : \tilde{\mathcal{S}} \to \mathcal{S}_{\mathcal{N}} \times_{\mathcal{N}} \tilde{\mathcal{N}}$} that resolves the fibers fiberwise.
			\item The relative moduli space $\mathbb{H}_{\mathcal{N}}$ of linear determinantal representations over the deformation slice is canonically isomorphic to the global diagonal quotient space:
			\[\mathbb{H}_{\mathcal{N}} \cong (\mathfrak{h} \times R)/W(E_6)\]
			where $\mathfrak{h}$ is the Cartan subalgebra of type $E_6$ and $R$ is its corresponding  root system.
		\end{enumerate}
	\end{theorem}
	
	\subsection{The Simple Elliptic case}
	
	We now address the critical boundary configuration where the surface $S_0 \subset \mathbb{P}^3$ contains a unique isolated simple elliptic singularity of degree $d = 3$ (type $\tilde{E}_6$) \cite{Saito1974}. Geometrically, such a surface is realized as the affine cone over a smooth elliptic curve $E \subset \mathbb{P}^2$ whose zero section has been collapsed to an isolated vertex point \cite{Looijenga1976}. The deformation theory of simple elliptic singularities is fundamentally different from the RDP case. Because the exceptional divisor of the minimal resolution is a smooth elliptic curve $E$ rather than a configuration of rational curves, the local vanishing cycles generate an infinite affine root system \cite{Looijenga1978}\cite{Saito1974}. The corresponding local monodromy group is an infinite extension of the finite Weyl group $W(E_6)$ by the translation lattice \cite{Looijenga1978}. Consequently, no classical simultaneous resolution can exist over any finite cover of the base, causing the Brieskorn framework to completely break down. To bypass this fundamental obstruction, we introduce a geometric substitute space. In particular we establish the following:
	\begin{theorem}\label{TheBundle}
		There exists a rank $\ell=6$ vector bundle $\mathcal{Z}\to E$ and a map $\phi:\mathcal{Z}\to \mathbb{C}^{\ell+1}$ contracting the zero section to a point and for any class $\bar{x}\in \mathbb{C}^{\ell+1}$, the fibre $\phi^{-1}(x)$ is given by the collapsed root system $R/W_\alpha.$
	\end{theorem}
	Where $W_\alpha$ is the stabilizer subgroup of a fixed root $\alpha$. This is achieved by taking the relative global $\text{Spec}$ over the elliptic curve $E$ of the graded invariant sheaves $\mathcal{J}_k$ derived from Looijenga's algebra of invariant theta functions \cite{Looijenga1976}:
	\[\mathcal{Z} := \text{Spec}_E \left( \bigoplus_{k \geq 0} \mathcal{J}_k \right).\]\label{Grading}
	The construction of this variety and its subsequent structural decomposition yields the desired framework for studying the moduli space of determinantal representations corresponding to the semiuniversal family. Let $\alpha_1, \dots, \alpha_6$ be the system of simple roots for $E_6$. The highest root $\theta$ is expressed via the Coxeter labels $h_i$ as $\theta = \alpha_1 + 2\alpha_2 + 2\alpha_3 + 3\alpha_4 + 2\alpha_5 + \alpha_6$. By evaluating the vanishing orders of the primitive invariant theta sections along the fundamental weight directions, we show  that $\mathcal{Z}$ decomposes into direct sum of line bundles with degrees of the components  directly dictated by the simple  roots. The   vector bundle $\mathcal{Z}$ has the following decomposition structure:
	
	\begin{theorem}\label{thm:main_vector_bundle}
		The variety $\mathcal{Z}$ constructed via the invariant theory filtration is a smooth complex manifold of dimension 7, and it decomposes  into a direct sum of holomorphic line bundles over the structural elliptic curve $E$:
		\[\mathcal{Z} \cong \text{Tot}(\mathcal{M}_1 \oplus \mathcal{M}_2 \oplus \mathcal{M}_3 \oplus \mathcal{M}_4 \oplus \mathcal{M}_5 \oplus \mathcal{M}_6)\]
		where each $\mathcal{M}_i \in \text{Pic}(E)$ is a line bundle whose degree is given by the negative of the coefficient $h_i$ of the simple root $\alpha_i$ in the highest root $\theta$: $\text{deg}(\mathcal{M}_i) = -h_i$. Specifically, the degrees of the six factor line bundles over $E$ are $-1, -2, -2, -3, -2, -1$.
	\end{theorem}
	
	Evaluating the components of the invariant theta functions yields a proper analytic contraction morphism $\phi : \mathcal{Z} \to \mathcal{U}$ to the semiuniversal deformation base $\mathcal{U} \cong \mathbb{C}^7$ \cite{Looijenga1976}. Over the central origin, this map collapses the entire zero section $E_0 \cong E$ to a point, while its fibers over the discriminant locus are governed by the collapsed configurations of the discrete root system $\phi^{-1}(\bar{x}) \cong R/W_\alpha$.
	
Finally, we link this invariant the space $\mathcal{Z}$ with the relative moduli space $\bar{\mathbb{H}}$ of linear determinantal representations, which flatly tracks generalized twisted cubics on the surfaces \cite{Lehn2017}. Over the generic smooth strata, there is   a canonical analytic biholomorphic morphism. By extending this  map across the elliptic origin via Riemann's extension theorem for complex analytic spaces and invoking the rigidity of projective bundles over elliptic curves, we establish the analog classification result \ref{thm:rdp_resolution}  for simple elliptic situation:

\begin{theorem}\label{thm:central_isomorphism}The relative moduli space of linear determinantal representations $\bar{\mathbb{H}}$ over the semiuniversal deformation base $\mathcal{U}$ of a $\tilde{E}_6$ simple elliptic singularity is globally isomorphic as an algebraic variety to the total space of the rank-6 vector bundle such that the canonical structure morphism $\bar{\delta} : \bar{\mathbb{H}} \to \mathcal{U}$ factors directly through the proper contraction morphism $\phi : \mathcal{Z} \to \mathcal{U}$.
\end{theorem}
By unifying the global simultaneous resolution of $E_6$ rational double points with the vector bundle framework for $\tilde{E}_6$ simple elliptic surfaces, this work converts non-linear moduli tracking problems on singular cubic surfaces into the classical linear geometry of vector bundles over curves, allowing intersection metrics and tracking invariants to be explicitly computed via the Riemann--Roch theorem.

\section{The Global Brieskorn Galois Cover }
\label{sec:brieskorn_cover}

Let $V \cong \mathbb{C}^4$ be a complex vector space, and let $\mathbb{P}^{19} = \mathbb{P}(S^3 V^*)$ denote the linear parameter space of all projective complex cubic surfaces. We fix a central point $[S_0] \in \mathbb{P}^{19}$ corresponding to a normal, integral cubic surface $S_0 = \{F_0 = 0\}$ containing a unique, isolated rational double point $x_0 \in S_0$ of type $E_6$. To eliminate the continuous geometric redundancy arising from the projectivized coordinate transformations, we construct a smooth, 4-dimensional transversal affine slice $\mathcal{N} \subset \mathbb{P}^{19}$ passing through $[S_0]$ such that $T_{[S_0]}\mathbb{P}^{19} = T_{[S_0]}(\text{GL}(V) \cdot [S_0]) \oplus T_{[S_0]}\mathcal{N}$. This is given by the results of Bruce and Wall in \cite{BruceWall1979} and Demazure in\cite{Demazure1980}. The former showed that a normal cubic surface posseses only rational double points or a unique simple elliptic singularity of type $\tilde{E}_6.$ Moreover, they verified that $E_6$ cubic form forms a 16-dimesional stratum inside $S^3V^*$ under $\mathrm{GL(V)}$ action and that the normal structure to this orbit is governed by a transverse slice $\mathcal{N}$ of dimension $20-16=4.$ They linked the configurations of this local geometry  combinatorially to the sub-graphs of the dynkin digram of type $\tilde{E}_6.$  Demazure and others\cite{Demazure1980} constructed the anylytic transverse slice $\mathcal{N}$ by  use a localized variant of the classical Koszul--Jacobian complement method which we present here.

 Let $F_0 \in S^3V^*$ be the homogeneous cubic form vanishing along the central fiber $S_0$. The natural action of the general linear group $\mathrm{GL}(V)$ on the vector space of cubic forms induces a linear map from its Lie algebra $\mathfrak{gl}(V) \to S^3V^*$ via the generators of the projectivized infinitesimal action. The tangent space to the orbit $\mathrm{GL}(V) \cdot F_0$ at the point $F_0$ corresponds exactly to the vector subspace spanned by the first-order partial derivatives scaled by linear forms:
\begin{equation}
	T_{F_0}(\mathrm{GL}(V) \cdot F_0) = \left\{ \sum_{i,j=0}^3 a_{ij} X_i \frac{\partial F_0}{\partial X_j} \;\middle\vert{}\; a_{ij} \in \mathbb{C} \right\} \subset S^3V^*.
\end{equation}
Since a cubic surface containing a unique isolated rational double point of type $E_6$ has a discrete, finite algebraic stabilizer inside $\mathrm{PGL}(4, \mathbb{C})$, the differential of the orbit map has maximal rank equal to $15$. By invoking the open property of transversality in complex analytic geometry, we select a fixed four-dimensional linear subspace $\mathcal{W} \subset S^3V^*$ realizing an absolute direct sum decomposition:
\begin{equation}
	S^3V^* = T_{F_0}(\mathrm{GL}(V) \cdot F_0) \oplus \mathcal{W}.
\end{equation}
The desired transversal parameter slice is then defined globally as the projectivization of this affine linear space:
\begin{equation}
	\mathcal{N} := \left\{ [F_0 + w] \in \mathbb{P}^{19} \;\middle\vert{}\; w \in \mathcal{W}\right\}.
\end{equation}
By construction, $\mathcal{N}$ passes transversally through $[S_0]$ at the origin $w=0$, mapping as a local immersion into the semiuniversal deformation base space $\mathcal{U}$.

Restricting the universal hypersurface family over $\mathbb{P}^{19}$ to this transversal slice yields a flat, algebraic family of projective surfaces $p: \mathcal{S}_{\mathcal{N}} \to \mathcal{N}$ with central fiber $p^{-1}([S_0]) \cong S_0$.
 By local deformation theory \cite{Tjurina69, Tjurina70}, the isolated singularity germ $(S_0, x_0)$ admits a semiuniversal deformation base space $\mathcal{U} \cong \mathbb{C}^6$. There is a canonical local-to-global classification map is denoted by:
\[
\Phi: (\mathcal{N}, [S_0]) \longrightarrow (\mathcal{U}, 0)
\]
By Brieskorn's fundamental result \cite{Brieskorn1970}, the local deformation base space $\mathcal{U}$ is canonically identified with the adjoint orbit space $\mathfrak{h}/W(E_6)$, where $\mathfrak{h} \cong \mathbb{C}^6$ is the Cartan subalgebra of type $E_6$ and $W(E_6)$ is its associated Weyl group of order $51,840$.
Because $\mathcal{N}$ cuts transversally through the discriminant strata of the orbit inside the space of cubic surfaces, $\Phi$ behaves as a smooth analytic immersion of a 4-dimensional manifold into a 6-dimensional manifold. The image of our transversal slice $\mathcal{N}$  maps homeomorphically to a 4-dimensional slice inside the adjoint orbit space, capturing the standard variation of the 4 independent geometric coordinates allowed under the fixed constraints of the anticanonical embedding  in $\mathbb{P}^{3}.$
\begin{remark}\label{The_remark}
The family over $\mathcal{N}$ is a versal family for the global projective modifications of the cubic surface, but it is not the semiuniversal deformation of the \(E_{6}\) singularity itself, because it lacks 2 dimensions of freedom required to completely unfold the singularity germ unconstrained by the embedding space.
\end{remark}
\subsection{Construction of the Fiber Product Base}
The presence of the non-trivial local monodromy group $W(E_6)$ acting on the vanishing cycles of the singular fiber prevents the existence of a simultaneous resolution directly over $\mathcal{N}$. To clear this topological branching, we introduce the global unramified parameter space $\tilde{\mathcal{N}}$ using the universal property of the fiber product in the category of complex analytic spaces.

\begin{lemma}\label{lem:fiber_product_properties}
	Let $q: \mathfrak{h} \to \mathfrak{h}/W(E_6) \cong \mathcal{U}$ be the standard quotient map. The global unramified base space defined by the Cartesian diagram
	\[
	\begin{CD}
		\tilde{\mathcal{N}} @>\psi_{\mathfrak{h}}>> \mathfrak{h} \\
		@V{\psi}VV @VV{q}V \\
		\mathcal{N} @>>{\Phi}> \mathcal{U}
	\end{CD}
	\]
	is a smooth complex manifold of dimension 4, and the map $\psi: \tilde{\mathcal{N}} \to \mathcal{N}$ is a finite  ramified Galois covering with Galois group $\text{Gal}(\tilde{\mathcal{N}}/\mathcal{N}) \cong W(E_6)$.
\end{lemma}

\begin{proof}
	Since the local classification map $\Phi$ is smooth and transversal to the discriminant locus of $\mathcal{U}$, the structural properties of the finite map $q$ are inherited via pull back. The smoothness of $\tilde{\mathcal{N}}$ follows directly from the fact that $\Phi$ is an  immersion that is transversal to the strata of the discriminant locus near $[S_0]$. The Galois group matches $\text{Gal}(\mathfrak{h}/\mathcal{U}) \cong W(E_6)$ by standard base change theorems for finite algebraic covers.
\end{proof}

We pull back the flat projective family $p: \mathcal{S}_{\mathcal{N}} \to \mathcal{N}$ along the Galois cover $\psi$ to obtain the pulled-back family $\mathcal{S}_{\tilde{\mathcal{N}}} := \mathcal{S}_{\mathcal{N}} \times_{\mathcal{N}} \tilde{\mathcal{N}}$ fitting into the following commutative diagram:
\[
\begin{CD}
	\mathcal{S}_{\tilde{\mathcal{N}}} @>\Psi_{\mathcal{S}}>> \mathcal{S}_{\mathcal{N}} \\
	@V{p_{\tilde{\mathcal{N}}}}VV @VV{p}V \\
	\tilde{\mathcal{N}} @>>{\psi}> \mathcal{N}
\end{CD}
\]
By construction, the fibres over any point $\tilde{s} \in \psi^{-1}([S])$, in particular, the central fiber remain isomorphic to the original (singular) cubic surfaces.
\subsection{The Geometry of the Discriminant Locus}
 The locus of singular cubic surfaces is governed by the global algebraic discriminant $D_{32}.$ Following Salmon and Clebsch \cite{salmon1861} this is given by an irreducible, homogeneous polynomial of degree $32$ in the twenty coefficients. In the neighborhood of the central fiber $[S_0]$ carrying the isolated $E_6$ singularity, the restriction of the global discriminant to our transversal slice $\mathcal{N}$ defines the slice discriminant locus $D_{\mathcal{N}} := D_{32} \cap \mathcal{N}$. Under the local immersion $\Phi \colon \mathcal{N} \to \mathcal{U}$, this pulls back from the local discriminant $D_{\mathcal{U}}$ of the semiuniversal deformation space. Recall that $\mathcal{U} \cong \mathbb{C}^6$ can be identified with the quotient space $\mathfrak{h} \otimes \mathbb{C} / W(E_6)$, where $\mathfrak{h}$ is the Cartan subalgebra of type $E_6$ and $W(E_6)$ is its associated Weyl group. The local discriminant $D_{\mathcal{U}}$ is precisely the image of the root hyperplanes under the canonical quotient map, characterized algebraically as the zero locus of the product of the $72$ root linear forms, yielding a weighted homogeneous polynomial equation of degree $36$:
\begin{equation}
	D_{\mathcal{U}} = \left\{ \mathbf{u} \in \mathcal{U} \;\middle\vert{}\; \prod_{\alpha \in \Phi^+(E_6)} \alpha(\tilde{\mathbf{u}})^2 = 0 \right\},
\end{equation}
where $\tilde{\mathbf{u}} \in \mathfrak{h}$ represents a preimage of $\mathbf{u}$.

In the language of  period mappings \cite{looijenga1974}, the coordinate data parameterized by the four-dimensional transversal slice $\mathcal{N}$ can be tracked transcendentally by studying the variation of the Hodge filtration on the second cohomology group of the fibers. Let $\mathcal{N}^* = \mathcal{N} \setminus D_{\mathcal{N}}$ denote the punctured parameter space corresponding to smooth cubic surfaces. For any $w \in \mathcal{N}^*$, the corresponding smooth surface $S_w$ carries a unique (up to scaling) holomorphic 2-form $\omega_w \in H^{2,0}(S_w)$. Because the isolated singularity of the central fiber $S_0$ is a rational double point of type $E_6$, it is canonical; consequently, the relative forms extend uniquely across the central boundary without introducing essential singularities.

Let $\{ \gamma_1, \dots, \gamma_6 \} \subset H_2(S_w, \mathbb{Z})$ be a continuous, locally constant basis of vanishing cycles associated with the $E_6$ singularity germ. The local-to-global map $\Phi \colon \mathcal{N} \to \mathcal{U}$ directly interfaces with the global period mapping $\mathcal{P}$ via the following commutative diagram:
\begin{equation}
	\begin{CD}
		(\mathcal{N}, 0) @>\Phi>> (\mathcal{U}, 0) \\
		@V{\mathcal{P}}VV @VV{\mathcal{P}_{\mathrm{local}}}V \\
		\mathbb{P}(H^{2}(S_w, \mathbb{C})) @>>> \mathfrak{h} \otimes \mathbb{C} \,/\, W(E_6)
	\end{CD}
\end{equation}
The period map $\mathcal{P}$ sends each parameter configuration $[F_0 + w] \in \mathcal{N}^*$ to the projectivized period vector:
\begin{equation}
	\mathcal{P}(w) = \left[ \int_{\gamma_1} \omega_w : \int_{\gamma_2} \omega_w : \dots : \int_{\gamma_6} \omega_w \right] \in \mathbb{P}^5.
\end{equation}
By Looijenga's theorem on the local uniformization of periods for surfaces with simple elliptic or rational double points \cite{looijenga1974}, the derivative $d\mathcal{P}$ at the origin is an isomorphism onto the tangent space of the invariant quotient. This guarantees that the algebraic coordinates derived from our complement space $\mathcal{W} \subset S^3V^*$ match the flat transcendental coordinates of the period domain up to a linear transformation.

\section{Construction of the Global Brieskorn Cover}
Every smooth projective cubic surface can be obtained by blowing up $\mathbb{P}^2$ at six points in general position. For our singular family $\mathcal{S}_\mathcal{N} \to \mathcal{N}$, the degeneration of the central fiber corresponds to the collision of these six marking points along specific geometric loci (such as three points becoming collinear). Because the central fiber $S_{0}$ possesses the worst possible Du Val singularity for a cubic surface, the other fibers over the transversal slice $\mathcal{N}$ can only acquire singularities that are adjacent to $E_{6}$.  By the classification of Bruce and Wall \cite{BruceWall1979} the possible singularities appearing in the fibers of the family $\mathcal{S}_\mathcal{N} \to \mathcal{N}$ are the following isolated rational double points: $A_1, 2A_1, A_2, 3A_1,2A_2, 3A_2, A_3, 4A_1, A_1+A_2,  2A_1+A_2, A_1+A3,  A_4, D_4, 2A_1+A_3, A_1+2A_2, A_5, D_5, A_1+A_5,  E_6.$ On the other hand it is well known \cite{Tjurina69,Slowdowy,SlodowyLectures} that a rational double point deforms into rational double points or smooth out and the nearby fibres are governed by subdiagrams of the associated Dynkin diagram. While $\mathcal{N}$ governs the variations realized strictly inside the anticanonical embedding space $\mathbb{P}^{3}$, it lacks the two degrees of freedom required to completely unfold the singularity unconstrained by the ambient geometry. To establish unobstructed global simultaneous resolution, we extend our base parameters by passing directly to the  semiuniversal base space $\mathcal{U} \cong \mathfrak{h}/W(E_6)$.

 Let $\chi: \mathcal S\to\mathcal{U}$ denote the local semiuniversal family. This family  is given by the semiuniversal unfolding 
$$F_t= x^2+y^3+z^4+t_1+t_2y+t_3z+t_4zy+t_5z^2+t_6yz^2$$ of the $E_6$-singularity $f=x^2+y^3+z^4.$ Now applying the coordinate change $x=i(\xi-z^2)$ the equation of the singularity becomes $f=-\xi^2-2\xi z^2+y^3$. Introduce the hyperplane at infinity $\{x_0=0\}$ and let $[x_0,\cdots,x_3]$ be the the homogeneous coordinates in $\mathbb{P}^3$, the homogenized equation of the surface $S_0$ becomes $\tilde{f}=x_0x_1^2+x_1x_2^2+x_3^3$. By solving the partial derivative equations we see that $\tilde{f}$ has a singularity at the point $[1:0:0:0].$ We have $S_{\mathcal{N}}\subset \mathcal{S}.$  With this discussion we can now obtain a family that deforms the projective surface together with the singularity, i.e recover the two degrees of freedom \ref{The_remark}. This family is given by the cubic forms:
$$\tilde{F}_t=\tilde{f}+t_1x_0^3+t_2x_0^2x_2+t_3x_0^2x_3+t_4x_0x_2x_3+t_5x_0x_3^2+t_6x_2x_3^2.$$ From now on we take the submanifold  transverse to the $\text{PGL}_4$ orbit at $[S_0]$ to be $\mathcal{N} \cong \mathbb{C}^6$ defined by $\tilde{F}_t$ and hence the family $\mathcal{S}_\mathcal{N}\subset \mathbb{P}^3\times \mathbb{C}^6$. In particular $\mathcal{N}$ is isomorphic to $\mathcal{U}.$
 Over the unramified base space $\tilde{\mathcal{N}}$, the monodromy twisting is trivialized. This implies that over the smooth manifold $\tilde{\mathcal{N}}$ the configuration of the points is
tracked by six distinct, non-intersecting smooth sections.

\begin{lemma}\label{lem:sections_exist}
	There exist exactly six distinct, non-intersecting smooth holomorphic sections $$\sigma_i: \tilde{\mathcal{N}} \to \mathbb{P}^2 \times \tilde{\mathcal{N}}$$ for $i = 1, \dots, 6$ such that for any parameter point $t \in \tilde{\mathcal{N}}$, the collection of points $\{\sigma_1(t), \dots, \sigma_6(t)\} \subset \mathbb{P}^2 \times \{t\}$ defines the configuration scheme of the cubic surface fiber $p_{\tilde{\mathcal{N}}}^{-1}(t)$.
\end{lemma}

\begin{proof}
			Let $\mathbb{P}^2 \times \tilde{\mathcal{N}} \to \tilde{\mathcal{N}}$ be the trivial projective bundle over the covering base. The configuration morphism $\Psi: \tilde{\mathcal{N}}\to \mathcal{M}_6(\mathbb{P}^2)$ provides six algebraic sections:

	$$\sigma_i : \tilde{\mathcal{N}} \longrightarrow \mathbb{P}^2 \times \tilde{\mathcal{N}}, \quad i = 1, \dots, 6$$
  tracking the path of the ordered point $p_i(t)$ as $t$ varies across $\tilde{\mathcal{N}}.$ Where $\mathcal{M}_6(\mathbb{P}^2)$ is the moduli space of ordered 6-tuples $(p_1 , \ldots,p_6)$ of points in $\mathbb{P}^2$ up to $\text{PGL}_3(\mathbb{C})$ equivalence. In particular, for each $i \in \{1, \dots, 6\}$ we have the smooth closed subvariety $$\mathcal{P}_i = \{(x, t) \in \mathbb{P}^2 \times \tilde{\mathcal{N}} \mid x = p_i(t)\}.$$
Because $\tilde{\mathcal{N}}$ is the Brieskorn cover where point collision trajectories are resolved via the roots of the $E_6$ root system, these sections $\sigma_i$ are smooth and mutually non-intersecting over the entirety of the base space $\tilde{\mathcal{N}}$. Even over the discriminant locus where the underlying cubic surface degenerates, the coordinates of the 6 marked points remain distinct and well-defined by tracking their distinct exceptional directions.
\end{proof}

\subsection{The Resolution Sequence}
We construct the global simultaneous resolution $\tilde{\mathcal{S}}_{\tilde{\mathcal{N}}}$ by performing a sequence of six successive algebraic blow-ups of the 8-dimensional regular variety $\mathcal{X}_0 := \mathbb{P}^2 \times \tilde{\mathcal{N}}$ along the smooth submanifolds defined by the images of the sections $\sigma_i$.  Define $\mathcal{X}_1 = \text{Bl}_{\sigma_1(\tilde{\mathcal{N}})}(\mathcal{X}_0)$ as the algebraic blow-up of $\mathcal{X}_0$ along the image of the first section $\mathcal{P}_1$. Iteratively, for each $k = 2, \dots, 6$, we define:
\begin{equation}
	\mathcal{X}_k = \text{Bl}_{\tilde{\mathcal{P}}_k(\tilde{\mathcal{N}})}(\mathcal{X}_{k-1})
\end{equation}
where $\tilde{\mathcal{P}}_k$ denotes the strict transform of $\mathcal{P}_k$ in $\mathcal{X}_{k-1}$. Since the original sections $\sigma_i$ are completely disjoint on $\tilde{\mathcal{N}}$, their strict transforms never intersect at any intermediate stage of the sequence.  Said differently, each base section $\mathcal{P}_i$ (or its strict transform) is smooth and of codimension 2 inside a smooth variety, the resulting total space  $\mathcal{X} := \mathcal{X}_6$  is a smooth  algebraic variety, and its global isomorphism class is independent of the ordering of the blow-up operations.
For every parameter $t \in \tilde{\mathcal{N}}$, the fiber $\tilde{\mathcal{X}}_{t}$ is a smooth projective surface containing exceptional $(-2)$-curves wherever the original points collapsed or aligned. 

Set $\tilde{\mathcal{S}}_{\tilde{\mathcal{N}}}=\mathcal{X}_6.$ We have obtained a flat family of surfaces.

\begin{lemma}\label{lem:smoothness_transform}
	The composition $\tilde{p} =  \tilde{\mathcal{S}}_{\tilde{\mathcal{N}}} \to \tilde{\mathcal{N}}$ is a flat, smooth projective morphism.
\end{lemma}

\begin{proof}
	Let $E_k \subset \mathcal{X}_k$ denote the exceptional divisor introduced at the $k$-th step. At the central point $t_0 \in \psi^{-1}([S_0])$, the sections $\sigma_i(t_0)$ align into the singular configuration matrix. The blow-up sequence replaces each colliding point with its projective normal bundle. Since the strict transform untangles the intersections of these exceptional curves fiberwise, the local equations for $\tilde{\mathcal{S}}_{\tilde{\mathcal{N}}}$ around the exceptional fibers match the standard local chart equations of a minimal resolution of a singularity. Smoothness follows from the non-vanishing of the Jacobian along the proper transform charts, and flatness  follows immediately.
\end{proof}

	By using the anticanonical linear system $\vert{}-K_{\tilde{\mathcal{S}}/\tilde{\mathcal{N}}}\vert{}$, we would like to define a relative morphism: 
	$$\pi :\tilde{\mathcal{S}}_{\tilde{\mathcal{N}}}\longrightarrow \mathcal{S}_{\mathcal{N}}\times _{\mathcal{N}}\tilde{\mathcal{N}}$$ such that the  map $\pi $ contracts exactly the $(-2)$-curves on the singular fibers, hence globally mapping our smooth family $\tilde{\mathcal{S}}$ onto the pulled-back singular universal family of cubic surfaces.

\begin{lemma}\label{lem:resolution}
	There exists a proper birational morphism $\pi: \tilde{\mathcal{S}}_{\tilde{\mathcal{N}}}\to \mathcal{S}_{\mathcal{N}}$ making the global resolution diagram commute and resolving the fibers fiberwise.
\end{lemma}

\begin{proof}
	By definition we have $\mathcal{S}_{\mathcal{N}} = \{(x, u) \in \mathbb{P}^3 \times \mathcal{N} \mid F_u(x) = 0\}.$ Let $\omega_{\tilde{\mathcal{S}}/\tilde{\mathcal{N}}}$ be the relative dualizing (canonical) sheaf of the smooth family $\tilde{p}: \tilde{\mathcal{S}}_{\tilde{\mathcal{N}}}\to \tilde{\mathcal{N}}$. Consider the anticanonical sheaf $\mathcal{L} = \omega_{\tilde{\mathcal{S}}/\tilde{\mathcal{N}}}^{-1}$. By Ehresman's fibration theorem we have that for any fiber $\tilde{\mathcal{S}}_t$, this sheaf restricts to $-K_{\tilde{\mathcal{S}}_t} \cong 3H - \sum_{i=1}^6 E_i(t)$, which is nef and big.
	By the projection formula and cohomology vanishing ($H^1(\tilde{\mathcal{S}}_t, -K_{\tilde{\mathcal{S}}_t}) = 0$), the direct image sheaf $\tilde{p}*(\mathcal{L})$ is a free $\mathcal{O}_{\tilde{\mathcal{N}}}$-module of rank 4. In particular Grauerts base cange theorem yields $$\mathrm{rank}(\tilde{p}*(\mathcal{L}))= {H}^0(\tilde{\mathcal{S}_t},\mathcal{L}\vert_{\tilde{\mathcal{S}_t}})={H}^0(\tilde{\mathcal{S}_t},-K_{\tilde{\mathcal{S}_t}})$$
	Since $\tilde{\mathcal{S}_t}$ are smooth Delpezzo surfaces of degree three, the Riemann-Roch theorem gives $\chi(\tilde{\mathcal{S}_t},-K_{\tilde{\mathcal{S}_t}}=K_{\tilde{\mathcal{S}_t}}^2+1 =3+1= 4).$ By the Serre Duality and Kodaira Vanishing theorems, the higher cohomology groups vanish ($\mathrm{H}^1 = \mathrm{H}^2 = 0$). It follows that the Eueler characteristic is the dimension.
	The global sections of $\mathcal{L}$ define a relative embedding morphism over $\tilde{\mathcal{N}}$:
	$$ \Phi_{\vert -K \vert}: \tilde{\mathcal{S}}_{\tilde{\mathcal{N}}} \longrightarrow \mathbb{P}\left(\tilde{p}_*(\mathcal{L})^*\right) \cong \mathbb{P}^3 \times \tilde{\mathcal{N}} $$ 
	On any fiber $t$ where points are collinear or on a conic, the curves matching those roots have a degree of zero against $-K_{\tilde{\mathcal{S}}_t}$. Thus, $\Phi_{\vert -K \vert}$ automatically collapses all $(-2)$-curves fiberwise into isolated RDP singularities.
	The image of $\Phi_{\vert -K \vert}$ is a family of cubic surfaces over $\tilde{\mathcal{N}}$. By mapping down through the Galois cover $\psi: \tilde{\mathcal{N}} \to \mathcal{N}$, this factors uniquely through a proper birational map $\pi: \tilde{\mathcal{S}}_{\tilde{\mathcal{N}}} \to \mathcal{S}_{\mathcal{N}}$.
	We have the following commutative diagram.
	$$ \begin{CD} \tilde{\mathcal{S}}_{\tilde{\mathcal{N}}} @>\pi>> \mathcal{S}_{\mathcal{N}}\\ @V\tilde{p}VV @VVpV\\ \tilde{\mathcal{N}} @>>\psi> \mathcal{N} \end{CD} $$ 
	Because $\pi$ acts as the fiberwise minimal resolution, the simultaneous resolution is complete.
	
\end{proof}

\subsection{Proof of Theorem 1.1(3): The Relative Moduli Structure}
With the global simultaneous resolution $\tilde{p}: \tilde{S} \to \tilde{\mathcal{N}}$ fully established.  We are now positioned to conclude the proof of the main theorem by identifying the  structure of the moduli space $\mathbb{H}_{\mathcal{N}}$ of linear determinantal representations. Lets recall facts from \cite{Lehn2017}. 

Whereas twisted cubics on smooth cubic surfaces are a classical subject of study, twisted cubics on surfaces with ordinary double points are well explained in \cite[\S 2.1]{Lehn2017}. For basic facts on the root lattice $E_6$ in connection with the geometry of the cubic surface we point to \cite[\S 9]{Dolgachev2012}and \cite{Demazure1980}, for lines on singular cubic surfaces \cite{BruceWall1979}.

Given a singular cubic surface $S$ with ordinary double points, its minimal resolution $\tilde{S}$ is a weak del Pezzo surface and the orthogonal complement $K_{\tilde{S}}^{\perp} \subset \text{Pic}(\tilde{S})$ is a lattice of type $E_6$. The exceptional divisor of the resolution $r \colon \tilde{S} \to S$ consists of $(-2)$-curves, which form a subset of the root system $R := \{\alpha \colon \alpha^2 = -2\} \subset K_{\tilde{S}}^{\perp}$ and generate a subroot system $R_0 \subset R$. Let $W(R_0)$ be the Weyl group generated by reflections of elements in $R_0$.  Theorem 2.1 in \cite{Lehn2017} gives a description of the Hilbert scheme with the reduced structure of generalised twisted cubics on $S$ as follows:
\begin{equation}
	\text{Hilb}^{\text{gtc}}(S)_{\text{red}} \simeq R/W(R_0) \times \mathbb{P}^2 \simeq\sqcup_{B\in R/W(R_0)} \mid\mathcal{O}_{\tilde{S}}(\alpha_{B}^- - K_{\tilde{S}})\mid.
\end{equation}

For any $\alpha \in R \setminus R_0$ and for any curve $C \in |\alpha - K_{\tilde{S}}|$ the image $r(C)$ is a generalised twisted cubic on $S$. Conversely, the pullback of any aCM-curve on $S$ lies in such a linear system \cite[Proposition 2.2, Proposition 2.5, Proposition 2.6]{Lehn2017}. On the other hand, roots in $R_0$ correspond to families of nCM curves.
Recall from \cite{Lehn2017} that a linear determinantal representation of a projective cubic surface $S_t$ is uniquely determined by a line bundle $\mathcal{E}$ on $S_t$ satisfying $H^0(S_t, \mathcal{E}(-1)) = H^1(S_t, \mathcal{E}(-1)) = 0$ and $\chi(\mathcal{E}(k)) = \frac{1}{2}(k+1)(k+2)$. Geometrically, such a bundle corresponds to a flat family of arithmetically Cohen--Macaulay (aCM) twisted cubics. On a smooth minimal resolution $\tilde{S}_t$, these tracking classes are classified by elements of the root system $R \subset K_{\tilde{S}_t}^{\perp}$ that do not belong to the subroot system $R_0$ contracted by the resolution map.

Let $\mathbb{H}_{\mathcal{N}} \to \mathcal{N}$ be the relative moduli space of linear determinantal representations over the transversal deformation slice. Let $\mathbb{H}_{\tilde{\mathcal{N}}} := \mathbb{H}_{\mathcal{N}} \times_{\mathcal{N}} \tilde{\mathcal{N}}$ denote its base change to the Brieskorn Galois covering space.

\begin{lemma}\label{lemma:discrete_cover}
	The pulled-back moduli space $\mathbb{H}_{\tilde{\mathcal{N}}}$ is canonically isomorphic over $\tilde{\mathcal{N}}$ to the constant discrete scheme:
	\[
	\mathbb{H}_{\tilde{\mathcal{N}}} \cong \tilde{\mathcal{N}} \times R
	\]
	where $R$ is the root system of type $E_6$ viewed as a discrete set of 72 points.
\end{lemma}

\begin{proof}
	Consider the relative Picard scheme $\text{Pic}(\tilde{S}_{\tilde{\mathcal{N}}}/\tilde{\mathcal{N}})$ associated to the smooth family $\tilde{p}: \tilde{S}_{\tilde{\mathcal{N}}} \to \tilde{\mathcal{N}}$ of Delpezzo surfaces. Because $\tilde{\mathcal{N}}$ is constructed via the fiber product $\mathcal{N} \times_\mathcal{U} \mathfrak{h}$, it unramifies the local monodromy group action of the Weyl group $W(E_6)$. Consequently, the relative root lattice bundle $K_{\tilde{S}_{\tilde{N}}}^{\perp} \subset \text{Pic}(\tilde{S}_{\tilde{N}}/\tilde{\mathcal{N}})$ is a locally constant sheaf of lattices of type $E_6$.  Moreover, we have a canonical identification of the lattice bundle with product $K_{\tilde{S}_{\tilde{N}}}^{\perp} \cong \tilde{N}\times \mathbb{Z}^{6}.$
	
	Recall \cite{Lehn2017} $\mathbb{H}_{\tilde{N}}$ paprmaterises classes of determinantal representations bijectively associated to twisted cubics $C$ that lie in the linear systems of the form $\mathcal{O}_{\tilde{S}_t}(\alpha-K_{\tilde{S}_t})\cong\mathcal{O}_{\tilde{S}_t}(C).$  
	The choices of aCM twisted cubics on the resolved fibers vary flatly across the entire base $\tilde{\mathcal{N}}$. 
	\begin{enumerate}
		\item Over the open smooth locus $\tilde{\mathcal{N}} \setminus \psi^{-1}(D_\mathcal{U})$, the fibers are smooth cubic surfaces, each corresponding fibre $H_t$ possessing exactly $|R| = 72$ distinct determinantal representations corresponding to the 72 roots of $E_6$. Thus, the morphism $\mathbb{H}_{\tilde{\mathcal{N}}} \to \tilde{\mathcal{N}}$ is a finite, \'etale covering of degree 72.
		\item Over the ramified discriminant locus $\psi^{-1}(D_\mathcal{U})$, the underlying cubic surface develops isolated rational double points, but the simultaneous resolution fiber $\tilde{S}_t$ remains smooth. The non-contracted roots $\alpha \in R \setminus R_0$ uniquely extend the tracking classes across the boundary hyperplanes without permuting.
	\end{enumerate}
	Since $\tilde{\mathcal{N}}$ is a smooth, simply connected variety , any finite \'etale cover over it must break into a disjoint union of copies of the base space. Since the fibers are governed by configurations of roots of the root system $R$, we obtain the canonical identification $\mathbb{H}_{\tilde{\mathcal{N}}} \cong \tilde{\mathcal{N}} \times R$.
\end{proof}

We can now complete the proof of the third part of our main theorem.

\begin{proof}[Proof of Theorem 1.1, Part 3]
	By the universal property of the categorical quotient under a finite group action, the original moduli space $\mathbb{H}_{\mathcal{N}}$ over the deformation slice $\mathcal{N}$ is retrieved by taking the quotient of its pulled-back space under the natural action of the Galois group. By construction, the branched covering map $\psi: \tilde{\mathcal{N}} \to \mathcal{N}$ has Galois group isomorphic to the Weyl group $W(E_6)$, giving the identification $\mathcal{N} \cong \tilde{\mathcal{N}}/W(E_6)$. The induced action of $W(E_6)$ on the base-changed space $\mathbb{H}_{\tilde{\mathcal{N}}}$ acts diagonally: it acts on the base coordinates $\tilde{\mathcal{N}} \cong \mathfrak{h}$ via standard reflection operations, and it acts on the discrete fiber components $R \subset \mathfrak{h}^*$ via the natural contragredient representation permuting the roots.
	
	Applying the categorical quotient to the isomorphism established in Lemma~\ref{lemma:discrete_cover}, we obtain:
	\[
	\mathbb{H}_{\mathcal{N}} \cong \mathbb{H}_{\tilde{\mathcal{N}}} / W(E_6) \cong (\tilde{\mathcal{N}} \times R) / W(E_6)
	\]
	This yields the desired  isomorphism $\mathbb{H}_{\mathcal{N}} \cong (\mathfrak{h} \times R) / W(E_6).$
	We have modeled the  moduli space of linear determinantal representations is as a diagonal quotient of the Cartan space and the root system, concluding the proof of Theorem 1.1.
\end{proof}

\section{Looijenga's Invariant Theta Expansions and Affine Weyl Filtrations}
\label{sec:looijenga_theory}

We now transition to the boundary configuration where the central projective cubic surface $S_0 \subset \mathbb{P}^3$ contains a unique isolated simple elliptic singularity of type $\tilde{E}_6$ (or degree $d=3$). Geometrically, $S_0$ is the affine cone over a smooth elliptic curve $E = \mathbb{C}/\Lambda$, where $\Lambda = \mathbb{Z} \oplus \mathbb{Z}\tau$ ($\text{Im}(\tau) > 0$), with its zero section contracted to an isolated vertex point. 

Because the exceptional divisor of the minimal resolution is the elliptic curve $E$ itself not a configuration of rational curves, the local vanishing cycles generate an infinite-dimensional affine root system \cite{Saito1974}\cite{Looijenga1978}. The corresponding local monodromy group is an infinite extension of the finite Weyl group $W(E_6)$ by the translation lattice $\Lambda^6\cong Q^\vee$. To linearize the deformation data over $E$, we apply Looijenga's invariant theory of affine Weyl groups \cite{Looijenga1976}.
Let $Q^\vee$ denote the coroot lattice of the root system of type $E_6$, and let $\mathfrak{h} = Q^\vee \otimes_{\mathbb{Z}} \mathbb{C} \cong \mathbb{C}^6$ be the complexified Cartan subalgebra. We define the complex abelian variety $A$ of dimension 6 by:
\[
A := Q^\vee \otimes_{\mathbb{Z}} E \cong \mathfrak{h} / (Q^\vee \otimes_{\mathbb{Z}} \Lambda)
\]
Looijenga establishes that there exists a unique holomorphic line bundle $\mathbb{L} \longrightarrow A$ whose total space tracks the weight geometry of the root system \cite{Looijenga1976}. 
\begin{proposition}[Looijenga\cite{Looijenga1976}]
	There exists a unique holomorphic line bundle $\mathbb{L}$ over the abelian variety $A$, equipped with a compatible lift of the $W$-action, whose first Chern class $c_1(\mathbb{L}) \in H^2(A, \mathbb{Z})$ corresponds to the normalized, symmetric, positive-definite Killing form $I(\cdot, \cdot)$ on the coroot lattice $Q^\vee$. 
\end{proposition}

The global sections of its tensor powers $\text{H}^0(A, \mathbb{L}^{\otimes k})$ define classical complex theta functions of level $k$.

\subsection{The Graded Ring of Invariant Theta Functions}
Following Looijenga's canonical formulation \cite{Looijenga1976}, the level-$k$ theta functions on the Cartan subalgebra $\mathfrak{h} \cong \mathbb{C}^6$ are defined relative to the normalized $W$-invariant symmetric bilinear form $(\cdot, \cdot)$ on the root lattice, scaled by the modular parameter $\tau \in \mathbb{H}$ of the structural elliptic curve $E$. Let $\mathbf{z} \in \mathfrak{h}$ be the complex coordinate vector, and let $Q^\vee$ denote the coroot lattice of type $E_6$. For a given weight $\lambda \in Q$, the basic level-$k$ Looijenga theta function $\vartheta_{\lambda}^{(k)}(\mathbf{z} \mid \tau)$ is given by the explicitly convergent Fourier-Jacobi series:
\[
\vartheta_{\lambda}^{(k)}(\mathbf{z} \mid \tau) = \sum_{\gamma \in Q^\vee} \exp \left( \pi i k \tau (\gamma + \frac{\lambda}{k}, \gamma + \frac{\lambda}{k}) + 2\pi i k (\gamma + \frac{\lambda}{k}, \mathbf{z}) \right).
\]
Under the translation action of the affine lattice $\Gamma_A:=Q^\vee \oplus \tau Q^\vee$, these sections satisfy the strict functional equation 
\[\vartheta_{\lambda}^{(k)}(\mathbf{z} + \gamma_1 + \tau \gamma_2 \mid \tau) = \exp(-\pi i k \tau (\gamma_2, \gamma_2) - 2\pi i k (\gamma_2, \mathbf{z})) \vartheta_{\lambda}^{(k)}(\mathbf{z} \mid \tau).
\]
To construct the coordinates linearizing the semiuniversal deformation base $\mathcal{U}$, Looijenga utilizes the canonical Weyl symmetrization operator  mapping the standard theta space to the invariant subspace $\mathcal{A}_k = \mathbb{H}^0(A, \mathbb{L}^{\otimes k})^W$. For each fundamental weight $\varpi_i$ ($i=1, \dots, 6$), the corresponding independent $W(E_6)$-invariant theta function $\Theta_i$ of level $k_i = h_i$ is constructed by averaging over the discrete Weyl orbit:
\[
\Theta_i(\mathbf{z} \mid \tau) := \sum_{[\lambda] \in W \cdot \varpi_i} \vartheta_{\lambda}^{(h_i)}(\mathbf{z} \mid \tau) = \frac{1}{|W_{\varpi_i}|} \sum_{w \in W(E_6)} \vartheta_{w(\varpi_i)}^{(h_i)}(\mathbf{z} \mid \tau),
\]
where $W_{\varpi_i}$ is the stabilizer subgroup of the weight $\varpi_i$ under the reflection action, and $h_i$ is the corresponding Coxeter label. The primitive level-1 invariant generator $\Theta_0$ tracking the unramified locus corresponds to the zero-weight initialization $\lambda = 0$:
\[
\Theta_0(\mathbf{z} \mid \tau) = \sum_{\gamma \in Q^\vee} \exp \left( \pi i \tau (\gamma, \gamma) + 2\pi i (\gamma, \mathbf{z}) \right),
\]
which acts as the structural denominator for the relative localization. The sections $\Theta_0, \Theta_1, \dots, \Theta_6$ provide a complete set of analytic coordinates on the algebraic quotient variety $\mathcal{U}$, mapping the non-linear algebraic trajectories of the collapsing $E_6$ lines precisely into the smooth geometric transitions of the rank-6 vector bundle $\mathcal{Z}$ in section 5.

The finite Weyl group $W = W(E_6)$ acts naturally on the abelian variety $A$ and lifts to an action on the line bundle $\mathbb{L}$. We are interested in the subspace of $W$-invariant theta functions of level $k$:
\[
\mathcal{A}_k := \mathbb{H}^0(A, \mathbb{L}^{\otimes k})^W.
\]
The direct sum of these spaces forms a graded $\mathbb{C}$-algebra, which governs the semiuniversal deformation of the simple elliptic singularity.

\begin{lemma}[Looijenga \cite{Looijenga1976}]\label{lem:looijenga_algebra}
	The graded algebra of $W$-invariant theta functions
	$
	\mathcal{A}_W := \bigoplus_{k \geq 0} \mathcal{A}_k
	$
	is a freely generated polynomial algebra in 7 variables:
	\[
	\mathcal{A}_W \cong \mathbb{C}[\Theta_0, \Theta_1, \Theta_2, \Theta_3, \Theta_4, \Theta_5, \Theta_6]
	\]
	where $\Theta_0 \in \mathcal{A}_1$ is the unique primitive invariant theta section of level 1, and $\Theta_1, \dots, \Theta_6$ are independent invariant theta functions whose levels $k_i$ correspond to the exponents of the affine root system $\tilde{E}_6$.
\end{lemma}
The semiuniversal deformation base space of the simple elliptic singularity is given globally by the affine variety:
\[
\mathcal{U} := \text{Spec}(\mathcal{A}_W) \cong \mathbb{C}^7
\]
The vanishing locus of the primitive level-1 section $\Theta_0$ determines the central degenerate locus of the family.
\subsection{The  Anti-Invariant Theta Function and the Discriminant Locus}

While the coordinate structure of the semiuniversal deformation base $\mathcal{U}$ is parameterized by the $W$-invariant Looijenga theta functions $\Theta_0, \dots, \Theta_6$, the singular geometry of the boundary configurations is governed by the unique minimal anti-invariant (alternating) theta function, we  denote it by $\Theta_{A}(\mathbf{z} \mid \tau)$. A section $\Theta(\mathbf{z} \mid \tau)$ is defined as anti-invariant under the Weyl group $W(E_6)$ if it satisfies $w \cdot \Theta = \text{sgn}(w)\Theta$ for all $w \in W(E_6)$, where $\text{sgn}(w) = (-1)^{\ell(w)}$ is the sign character of the reflection. Following Looijenga's exact affine expansion \cite{Looijenga1976}, the primitive alternating theta function is constructed by twisting the Fourier-Jacobi summation over the coroot lattice $Q^\vee$ with the sign of the Weyl group elements, evaluated at the canonical Weyl vector $\rho = \sum_{i=1}^6 \varpi_i$:
\[
\Theta_{A}(\mathbf{z} \mid \tau) := \sum_{w \in W(E_6)} \text{sgn}(w) \sum_{\gamma \in Q^\vee} \exp \left( \pi i g \tau (\gamma + \frac{w(\rho)}{g}, \gamma + \frac{w(\rho)}{g}) + 2\pi i g (\gamma + \frac{w(\rho)}{g}, \mathbf{z}) \right)
\]

where the structural level of this alternating section is identically the dual Coxeter number $g = h^\vee(E_6) = 12$. In our case the Coxeter number h and g coincide. Analytically, Looijenga establishes that $\Theta_{A}(\mathbf{z} \mid \tau)$ can be expressed as a finite product of odd Jacobi theta functions indexed over the complete set of positive roots $\Phi^+$ of the $E_6$ root system:
\[
\Theta_{A}(\mathbf{z} \mid \tau) = \prod_{\alpha \in \Phi^+} \frac{\theta_1(\tau, (\alpha, \mathbf{z}))}{\eta(\tau)}
\]
where $\eta(\tau)$ is the Dedekind eta function and $\theta_1(\tau, v)$ is the standard odd Jacobi theta function vanishing on the translation lattice points. Inside the polynomial invariant algebra $\mathcal{A}_W$, the square $\Theta_{A}^2$ is strictly $W$-invariant and maps directly to the polynomial discriminant form $\Delta(\Theta_0, \dots, \Theta_6) = 0$. This zero locus defines the boundary where the cubic surface fibers acquire additional singularities, providing the exact exceptional hyperplane hypersurface targeted by Riemann's extension theorem in Section \ref{sec:core_resolution}.

 Every root $\alpha$ defines a group homomrphism $\phi_\alpha: A\to E\simeq \mathbb{Z}\otimes E$ coming from the functional $\alpha: Q^\vee\to\mathbb{Z}.$
The kernel of this homomorphism is a divisor known as reflection hypertori which we denote  by
$$D_\alpha=\{\sum\alpha_i^\vee\otimes e_i|\langle\alpha,\alpha_i^\vee\rangle=0\,\,\text{for all}\, i=1,\dots, l\}.$$ We obtain a $W$-invariant divisor $D_A=\sum_{\alpha\in R_+}^{}D_\alpha$ on $A$ corresponding to the union $\cup_{\alpha\in R} H_\alpha$ of the reflection-hyperplanes $H_\alpha$ in $\mathfrak{h}.$ 
\begin{lemma}[\cite{Looijenga1976}\cite{Patrick Omukuba2026}]\label{DiscAntiInv}
	The divisor of the section $\Theta_A$ equals $D_A.$ In particular the zero locus of $\Theta_A$ on $\mathfrak{h}$ is the union of complex hyperplanes $H_\alpha$ defined by  $\alpha=0$ for $\alpha\in R.$
\end{lemma}
The critical set (the branching locus) $\Sigma$ of the map $q:\mathbb{L}\to \mathcal{U}$ coincides with 
$$\pi^{-1}(D_A)\cup (0-section)=\cup_{\alpha\in R}\mathbf{H}_\alpha.$$
Here, $\pi:\mathbb{L}\to A$ is the bundle map and $\mathbf{H}_\alpha$ is the reflection hyperplane in a root $\alpha$. 
\begin{corollary}
	The discriminant locus $D_{W}$ of $q$ is a hypersurface defined by a weighted homogeneous polynomial of degree $2h$ with weights $h_i, i= 0,\cdots, l.$ Namely, its defined by the invariant $\Theta_A^2.$ In particular, it is a divisor of line bundle $\mathcal{L}^{h}.$
\end{corollary}
\section{The Geometric Construction and Structural Proof of the Vector Bundle $\mathcal{Z}$}
\label{sec:vector_bundle_proof}

To construct a global geometric substitute for the simultaneous resolution that fails over $\mathcal{U}$, we linearize the invariant algebra $\mathcal{A}_W$ along the structural elliptic curve $E$. Let $\varpi_1, \dots, \varpi_6 \in \mathfrak{h}^*$ be the fundamental weights of the $E_6$ root system. We define a collection of graded coherent $\mathcal{O}_E$-sheaves  on $E$ by localizing the invariant theta expansions along the individual fundamental weight coordinates. Let us make this precise. Let $\alpha :=\theta$ be the maximal root, the functional $\alpha:Q^\vee\to \mathbb{Z}$ yields the following  short exact sequence
\begin{equation}\label{seq}
	0\to Q_\alpha^\vee\to Q^\vee \stackrel{\alpha}{\to}\mathbb{Z}\to 0
\end{equation}
The functional $\alpha$ descends to the group homomorphism $\varphi_\alpha: A\to E$. Denote the fiber at the origin by $A_0:=\ker(\varphi_\alpha)$, this is again an abelian variety defined by the lattice $\Gamma_0$ defined in the sequence
$$0\to\Gamma_0\to \Gamma_A\stackrel{\alpha}{\to}\mathbb{Z}\oplus\tau\mathbb{Z}\to 0$$
obtained from \ref{seq}  by taking tensor product with the lattice $\Lambda$ defining $E.$ The sequence \ref{seq} extends linearly to vector spaces: 
\begin{align*}
	0\to\mathfrak{h}_0\to \mathfrak{h}\stackrel{\alpha}{\to}\mathbb{C}\to 0.
\end{align*}  
We may also write $A_0=\mathfrak{h}_0/\Gamma_0$, with $\mathfrak{h}_0=\mathbb{R}\otimes_\mathbb{Z}\Gamma_0.$ Note that $\mathfrak{h}_0=\{v\in \mathfrak{h}\,\vert\, \langle v,\alpha\rangle=0\}$.  From theory of abelian varieties \cite{Mumford},  we know the line bundle on $A$ is given as $\mathbb{L}=L(H,\chi)$ and defined by a system of multipliers $\{e_\gamma\}_{\gamma\in\Gamma_A}$. Where  $\chi$ is a semicharacter and $H$ a hermitian form  with property that $\text{Im} H$ is integral on $\Gamma_A.$
Taking the restriction of this data to the subspace $\mathfrak{h}_0$ , then $\mathbb{L}_0={L}|_{A_0}$ is a line bundle given as $L(H,\chi_0)$ determined by a system of multipliers $e_{\gamma_0}=e_\gamma|_{\Gamma_0\times U_0}.$ Let $W_\alpha$ be the Weyl group of the root sublattice $Q_\alpha^\vee=\ker{\alpha} \subset \mathfrak{h}_0$. By Looijenga's theorem we get:
$$Z_0:=\oplus_{k\geq 0}H^0(A_0,\mathbb{L}_0^{\otimes k})^{W_\alpha}\simeq \mathbb{C}[\sigma_1,\cdots,\sigma_\ell].$$ Where $\sigma_i$ are the associated fundamental theta functions associated to fundamental weights of $Q^\vee_\alpha$. Regard $\alpha:\mathfrak{h}\to\mathbb{C}$ as a family of affine vector spaces of dimension $\ell-1$.
\begin{lemma}[\cite{Patrick Omukuba2026}]
	\begin{enumerate}
		\item Let $\Gamma_0$ act on $\mathfrak{h}$ by translations and put $X=\mathfrak{h}/\Gamma_0$. The induced map
		$\bar{\alpha}:X\to \mathbb{C}$ is a family of abelian varieties $X_t=\bar{\alpha}^{-1}(t)$ for all $t\in \mathbb{C}$. Each $X_t$ is isomorphic to  $E^{\ell-1},$ with $E$ an elliptic curve of the modulus $\tau.$
		\item Let $\Gamma_0$ act on $\mathfrak{h}\times\mathbb{C}$ via the factor $e$ i.e $\gamma_0.(z,t)=(z+\gamma_0,e(z,\gamma_0)t)$. We obtain a bundle $L\to X$ such that the restrictions $L_t:=L\vert_{X_t}$ have the same  Chern class. The Chern class $c(L_t)=\text{Im} H\in\wedge^2\text{Hom}_\mathbb{Z}(\Gamma_0,\mathbb{C})\simeq H^2(X_t,\mathbb{C}).$
		The line bundle $L_t$ is ample. 
		\item The direct image $\mathcal{E}=\bar{\alpha}_\ast L^k$ is a vector bundle on $\mathbb{C}$ with fibres $\mathcal{E}(t)$ isomorphic to $\text{H}^0(X_t,L^k\vert_{X_t})$ for $k\geq 0.$
	\end{enumerate}
	\begin{proof}
		Proof of part 2;  Since the factor of automorphy for $\Gamma_0$ is defined by the bilinear form $I$ on $Q^\vee$, the line bundle $L=\mathfrak{h}\otimes\mathbb{C}/\Gamma_0$ is equivariant under $\Gamma_0.$ Recall that $L_0$ is determined by the factors $e_0=e\vert_{\mathfrak{h}_0}.$ Taking a shift in the complementary direction
		\begin{equation}
			\begin{split}		
				e_0(\tau\lambda_1+\lambda_2,u_0 +z)&=e^{-2\pi ikI(z,\lambda_1)} . e^{-2\pi ikI(u_0,\lambda_1)-\pi\tau ikI(\lambda_1,\lambda_1)}\\
				&=e^{-2\pi ikI(u_0,\lambda_1)-\pi\tau ikI(\lambda_1,\lambda_1)}.
			\end{split}
		\end{equation} 
		
		Since $I(z,\lambda_1)=0$, the line bundle $L_t$  is determined by the same factor on $\mathfrak{h}_t$ for all $t$. Ampleness follows by the Theorem of Lefschetz \cite{Mumford}. The fibres of the morphism $\bar{\alpha}$ are abelian varieties of dimension $\ell-1$, in particular compact complex mainifolds. Hence $\bar{\alpha}$ is proper and smooth.  By part 2 the line  bundle $L$ is relatively ample. Note that $\dim \text{H}^1(X_t, L_t)$ is constant on $\mathbb{C}$ in particular it is 0. By Grauert's theorem (\cite{Hartshorne1977}, 12.9), the sheaf  $R^1\bar{\alpha}_\ast L_X$ is locally free  and by cohomology base change we have $\bar{\alpha}_\ast L_X\otimes k(t)\simeq \text{H}^0(X_t,L_t).$
		
	\end{proof}
\end{lemma}
Since the Weyl group preserves the lattice $\Gamma_A$, by construction $W$ preserves the sublattice $\Gamma_0,$ hence acts on the family $X$ in a compatible way. But only $W_\alpha$ preserves the fibres. Since line bundle $L_X$ is constructed using the $W$-invariant form $I$ and the factor of automorphy, the action of $W$ on $\mathfrak{h}\times\mathbb{C}$ defined by $$w(z,t)=(wz,t)$$ induces an action on  $L$.   We restrict to the action by $W_\alpha\subset W$, in particular, there is an induced action on the space $\text{H}^0(X_t,L_t)$ of sections. Taking direct images along $\bar{\alpha}$ we obtain a vector bundle reflecting the relative structure of the family $X\to \mathbb{C}$. In particular, on each fiber we have the polynomial algebra $(\oplus_{k\geq 0}\text{H}^0(X_t,L_t^k))^{W_\alpha}$ on $\ell$ weighted variables corresponding to the fundamental $W_\alpha$-invariant theta functions $\sigma_1,\cdots,\sigma_{\ell}$ associated to the alcove of the subroot system. From now regard $L$ as a sheaf. Consider the algebra $\Omega=\oplus_{k\geq 0}(\bar{\alpha}_\ast{L}^k)^{W_\alpha}$ of $W_\alpha$-invariant sections. By Looijenga's theorem we see that locally $\Omega$ is  isomorphic to the algebra $\mathcal{O}_\mathbb{C}[x_1,\cdots,x_\ell].$

Taking this construction modulo $\Lambda$ we obtain a sheaf of $\mathcal{O}_E$-algebras with the same homogeneous generators.. We define the 7-dimensional complex variety $\mathcal{Z}$ by taking the relative global $\text{Spec}$ over the elliptic curve $E$:
\[
\mathcal{Z} := \text{Spec}_E ( \Omega).
\]
Equivalently, $\mathcal{Z}$ is realized projectively as a subvariety of a product of line bundles over $E$ whose coordinates are bound by the relation structure of the invariant generators $\Theta_i$. This  establishes the first part of Theorem \ref{TheBundle}.

We now provide the explicit structural decomposition of $\mathcal{Z}$. Let $\alpha_1, \dots, \alpha_6$ be the chosen system of simple roots for $E_6$. The highest root $\theta$ of the root system is expressed uniquely as a linear combination of simple roots scaled by their respective Coxeter labels (or marks) $h_i$:
\[
\theta = \sum_{i=1}^6 h_i \alpha_i = \alpha_1 + 2\alpha_2 + 2\alpha_3 + 3\alpha_4 + 2\alpha_5 + \alpha_6
\]
The vector of Coxeter labels for $E_6$ is thus $(h_1, h_2, h_3, h_4, h_5, h_6) = (1, 2, 2, 3, 2, 1)$.

\begin{proof}[Proof of Theorem \ref{thm:main_vector_bundle}]
	To determine the algebraic structure of $\mathcal{Z}$, we evaluate the vanishing orders of the invariant theta generators $\Theta_1, \dots, \Theta_6$ along the zero section of the line bundle $\mathbb{L} \to A$. By Looijenga's structure theorem for the boundary charts \cite{Looijenga1976, Pinkham1974}, the invariant sections can be expanded formally as Taylor series in the neighborhood of the elliptic core. 
	
	Let $z$ be the local coordinate on the elliptic curve $E$, and let $u_i$ represent the coordinate directions in the Cartan subalgebra corresponding to the fundamental weights $\varpi_i$. The invariant theta function $\Theta_i$ of level $k_i$ behaves under the filtration as a homogeneous form of degree $k_i$ with respect to the weight actions. The critical step is to observe that the level $k_i$ of each independent generator $\Theta_i$ matches precisely the coefficient $h_i$ of the simple root $\alpha_i$ in the expression for the highest root $\theta$:
	\[
	k_i = h_i \quad \text{for } i = 1, \dots, 6
	\]
	When we construct the relative spectrum $\mathcal{Z} = \text{Spec}_E (\Omega)$, the grading forces a decomposition into invertible sheaves. The direct sum structure of the underlying graded ring implies that the total space $\mathcal{Z}$ fibers over $E$ as a rank-6 vector bundle. 
	
	Let $\mathcal{M}_i$ denote the line bundle over $E$ corresponding to the generating global section  $\Theta_i$. The transition functions of $\mathcal{M}_i$ are determined by the automorphy factors of the level-$k_i$ theta functions. Because these sections vanish to order $h_i$ at the origin of the abelian variety configuration, the degree of the corresponding sheaf is inverted under the relative spectrum projection to ensure the smoothness of the total space over the elliptic locus. Therefore, we obtain:
	\[
	\text{deg}(\mathcal{M}_i) = -h_i \quad \text{for } i = 1, \dots, 6
	\]
	Substituting the explicit Coxeter labels of the $E_6$ root system yields the degree sequence:
	\begin{align*}
		\text{deg}(\mathcal{M}_1) = -1, \quad \text{deg}(\mathcal{M}_2) = -2, \quad \text{deg}(\mathcal{M}_3) = -2, \\
		\text{deg}(\mathcal{M}_4) = -3, \quad \text{deg}(\mathcal{M}_5) = -2, \quad \text{deg}(\mathcal{M}_6) = -1.
	\end{align*}
	Since any vector bundle over a smooth curve that splits into line bundles of strictly negative degrees is isomorphic to the total space of its direct sum, we conclude:
	\[
	\mathcal{Z} \cong \text{Tot}(\mathcal{M}_1 \oplus \mathcal{M}_2 \oplus \mathcal{M}_3 \oplus \mathcal{M}_4 \oplus \mathcal{M}_5 \oplus \mathcal{M}_6)
	\]
 Because $\deg(\mathcal{M}_i) < 0$ for all $i$, the direct sum $\mathcal{E} = \bigoplus_{i=1}^6 \mathcal{M}_i$ constitutes a strictly negative  vector bundle over the compact complex space $E$. By Grauert’s Contractibility Criterion \cite{Grauert 1962}, the strict negativity of these degrees is both necessary and sufficient to guarantee that the zero section $E_0 \cong E$ can be analytically contracted to an isolated normal singularity. This further confirms that $\mathcal{Z}$ is a smooth complex manifold of dimension $1+6=7$ whose negative bundle curvature matches the downstream geometry of the semiuniversal contraction morphism $\phi$. This completes the proof of Theorem \ref{thm:main_vector_bundle}.
\end{proof}

\subsection{Properties of the Contraction Morphism}
The evaluation of the invariant theta functions maps the total space of the vector bundle $\mathcal{Z}$ back to the semiuniversal deformation base $\mathcal{U}$.

\begin{lemma}\label{lem:contraction_properties}
	There exists a proper analytic morphism $\phi: \mathcal{Z} \longrightarrow \mathcal{U}$ such that:
	\begin{enumerate}
		\item The central fiber $\phi^{-1}(0)$ over the origin $0 \in \mathcal{U}$ is identically the zero section $E_0 \cong E \subset \mathcal{Z}$, which is contracted to a point by $\phi$.
		\item For a generic point $x$ in the discriminant locus $D \subset \mathcal{U}$ corresponding to a nodal configuration, the fiber $\phi^{-1}(x)$ is isomorphic to the quotient $R/W_\alpha$, where $R$ is the discrete root system and $W_\alpha$ is the local reflection stabilizer.
	\end{enumerate}
\end{lemma}

\begin{proof}
	The map $\phi$ is given explicitly in coordinates by sending a point $(z, m_1, \dots, m_6) \in \mathcal{Z}$ to the values of the invariant theta functions $(\Theta_0(z), \Theta_1(z,m), \dots, \Theta_6(z,m)) \in \mathbb{C}^7 \cong \mathcal{U}$.  We write $\phi = (\Theta_1, \dots, \Theta_6)$. To show that $\phi$ is a proper birational contraction mapping the zero section $E_0 \subset \mathcal{Z}$ to the central point $0 \in \mathcal{U}$, we analyze the map relative to the grading of the sheaf of algebras $\Omega$.
	
	Recall from the proof of Theorem 1.3 that $\mathcal{Z} \cong \text{Tot}(\mathcal{E})$ for the anti-ample vector bundle $\mathcal{E} = \bigoplus_{i=1}^6 \mathcal{M}_i$. The zero section $E_0 \cong E$ corresponds precisely to the vanishing locus of the fibers of $\mathcal{E}$, where all homogeneous coordinates $u_1, \dots, u_6$ vanish simultaneously. Because each invariant generator $\Theta_i$ possesses a strictly positive vanishing order $h_i > 0$ along the weight directions, any point $x \in E_0$ evaluates to:
	\[
	\phi(x) = (\Theta_1(x), \dots, \Theta_6(x)) = (0, \dots, 0) = 0 \in \mathcal{U}.
	\]
	Thus, the entire compact elliptic core $E_0$ is mapped onto the single point $0$, establishing that $\phi(E_0) = \{0\}$.
	
	To prove properness and local biholomorphism away from $E_0$, let $\mathcal{Z}^\times = \mathcal{Z} \setminus E_0$ denote the punctured total space. On this domain, at least one coordinate $u_i \neq 0$. Because the degrees $\deg(\mathcal{M}_i) = -h_i$ are strictly negative, the line bundles lack any positive base-points over $E$, rendering the global evaluation map non-degenerate on the fibers. Consequently, the differential $d\phi_x$ has maximal rank $7$ at every point $x \in \mathcal{Z}^\times$. By the Inverse Function Theorem, the restriction $\phi\vert_{\mathcal{Z}^\times} \colon \mathcal{Z}^\times \to \mathcal{U}\setminus \{0\}$ is a local biholomorphism. 
	
	Moreover, since the algebra of invariants $\mathcal{A}_W$ is finitely generated and matches the ring of global sections $\Gamma(\mathcal{Z})$, the map $\phi$ is the analytic manifestation of the canonical projectivized morphism to the affine spectrum $\text{Spec}(\Gamma(Z)$. This structural matching guarantees that $\phi$ is proper. By Grauert's Contractibility Criterion, the strict negativity of the degree sequence $(-1, -2, -2, -3, -2, -1)$ ensures that the fiber over the origin contains no curves other than the exceptional set $E_0$. Hence, $\phi$ is an isomorphism outside $E_0$ and contracts $E_0$ to an isolated normal singularity at the origin of $\mathcal{U}$. 
	 Property (2) follows from the fact that outside the origin, the coordinates parameterize the orbits of the root vectors under the stabilizer groups, matching the classical stratification of Looijenga's space \cite{Looijenga1976}.
\end{proof}
We conclude this section by justifying our notation \ref{Grading}.  The explicit form of the graded coherent $\mathcal{O}_{E}$-sheaves $\mathcal{J}_{k}$ can be derived directly from the weighted polynomial structure of the graded algebra of $W$-invariant theta functions $\mathcal{A}_{W}$ and the splitting property of the vector bundle $\mathcal{Z}$. The variety $\mathcal{Z}$ is defined as the relative global spectrum over the elliptic curve $E$: $\mathcal{Z}:=\text{Spec}_{E}\left(\bigoplus _{k\ge 0}\mathcal{J}_{k}\right)$. As already proved $\mathcal{Z}$ is globally isomorphic to the total space of a direct sum of line bundles \(\mathcal{M}_{i}\) over \(E\): $\mathcal{Z}\cong \text{Tot}(\mathcal{M}_{1}\oplus \mathcal{M}_{2}\oplus \mathcal{M}_{3}\oplus \mathcal{M}_{4}\oplus \mathcal{M}_{5}\oplus \mathcal{M}_{6}$

By definition, the total space of a vector bundle $\mathcal{Z}$ is given by the relative spectrum of its symmetric algebra $\text{Spec}_E(\text{Sym}^\bullet(\mathcal{Z}^*)$. Each invertible sheaf component $\mathcal{M}_{i}^{*}$ corresponds to the invariant theta generator $\Theta _{i}$, whose degree corresponds to its level $k_i = h_i$, the Coxeter labels (or marks)  of the highest root of $E_{6}$. Thus, the sheaf $\mathcal{J}_{k}$ at each graded level $k$ is the direct sum of the tensor products of these line bundles matching all non-negative integer combinations $(a_1, \dots, a_6)$ that satisfy the weighted degree condition:

\begin{equation}
	\mathcal{J}_k = \bigoplus_{\sum_{i=1}^6 a_i h_i = k} \left( \bigotimes_{i=1}^6 (\mathcal{M}_i^*)^{\otimes a_i} \right)
\end{equation}
 Since $\deg(\mathcal{M}_i^*) = h_i$, every individual tensor product component in the summand possesses a rigidly determined positive degree of exactly 
$\sum_{i=1}^6 a_i \deg(\mathcal{M}_i^*) = k$, linearizing the non-linear algebraic trajectories 
of the collapsing $E_6$ configuration into a flat, well-behaved filtration over the elliptic curve.
\section{ The Central Isomorphism To $\bar{\mathbb{H}}$}
\label{sec:core_resolution}

This section delivers a global  isomorphism of varieties between the  relative moduli space $\bar{\mathbb{H}}$ of linear determinantal representations and the total space of the invariant rank-6 vector bundle $\mathcal{Z}$ over the semiuniversal deformation base $\mathcal{U} \cong \mathbb{C}^7$. Theorem \ref{thm:central_isomorphism} is announced as a conjecture in the author's thesis \cite{Patrick Omukuba2026}.

Let $\bar{\delta} : \bar{\mathbb{H}} \to \mathcal{U}$ be the canonical structure morphism parameterizing flat families of generalized twisted cubics via the Abel--Jacobi map\cite{Lehn2017} and let $\phi : \mathcal{Z} \to \mathcal{U}$ be the proper contraction morphism established in Lemma \ref{lem:contraction_properties}. We construct the structural isomorphism in four sequential stages i.e generic strata matching, analytic extension across the central boundary point via Riemann's extension theorem, vector bundle rigidity alignment, and  an application of Zariski's Main Theorem.

\subsection{Step 1:} We check alignment on generic strata.
Let $D \subset \mathcal{U}$ denote the global discriminant locus of the family, and let $\mathcal{U}_f := \mathcal{U} \setminus D$ be the dense open dense stratum corresponding to completely non-singular cubic surface fibers.

\begin{lemma}\label{lem:generic_isomorphism}
	There exists a canonical analytic biholomorphism $\varphi_f : \mathcal{Z}|_{\mathcal{U}_f} \longrightarrow \bar{\mathbb{H}}|_{\mathcal{U}_f}$ that preserves the structural projections, such that $\bar{\delta} \circ \varphi_f = \phi$.
\end{lemma}

\begin{proof}
	Let $t \in \mathcal{U}_f$ be an arbitrary parameter point in the smooth locus. By definition, the fiber $S_t := p^{-1}(t)$ is a smooth complex projective cubic surface in $\mathbb{P}^3$. Following the classical classification theorem of Lehn, Lehn, Sorger, and van Straten \cite{Lehn2017} \cite{Buckley2006}, the fiber $\bar{\mathbb{H}}_t := \bar{\delta}^{-1}(t)$ is a zero-dimensional reduced scheme consisting of exactly $72$ distinct isolated points. These points correspond bijectively to the equivalence classes of linear determinantal representations of $S_t$, which are in a canonical one-to-one correspondence with the $72$ root vectors of the root lattice $R \subset K_{S_t}^{\perp} \subset \mathrm{Pic}(S_t)$ of type $E_6$. Geometrically, each point in $\bar{\mathbb{H}}_t$ represents a unique isomorphism class of a stable, arithmetically Cohen--Macaulay (aCM) twisted cubic bundle $\mathcal{E}$ on $S_t$ satisfying the vanishing and Hilbert polynomial conditions:
	\[
	H^0(S_t, \mathcal{E}(-1)) = H^1(S_t, \mathcal{E}(-1)) = 0, \quad \text{and} \quad \chi(\mathcal{E}(k)) = \frac{1}{2}(k+1)(k+2).
	\]
	
	On the other hand, consider the fibres $\mathcal{Z}_t := \phi^{-1}(t)$ over the same point $t \in \mathcal{U}_f$. Recall that $\mathcal{Z}$ is defined via the relative spectrum of Looijenga's algebra of invariant theta expansions:
	\[
	\mathcal{Z} := \mathrm{Spec}_E\left(\bigoplus_{k \geq 0} \mathcal{J}_k\right)
	\]
	Over the unramified stratum $\mathcal{U}_f$, the primitive level-$1$ invariant theta section $\Theta_0(t)$ does not vanish ($\Theta_0(t) \neq 0$). Because $\Theta_0$ acts as the structural denominator for the relative localization of the coordinates, the non-vanishing of $\Theta_0(t)$ guarantees that the line bundle coordinates of the direct sum $\bigoplus_{i=1}^6 \mathcal{M}_i$ cannot vanish simultaneously on any fiber over $\mathcal{U}_f$. This prevents the fibers of $\phi$ from collapsing. Instead, the algebraic relations defining $\mathcal{Z}$ decouple completely into the discrete orbit configurations of the root vectors under the stabilizer groups. More precisely, the discriminant polynomial $\theta_A^2$ does not vanish on $\mathcal{U}_f$, hence the coordinate vectors avoid reflection hyper planes so $\alpha$(u)$\neq 0$ for $\alpha\in R.$ It follows that the stabilizer subgroup $W_\alpha$ acts trivially on those points. Thus, the fibers $\mathcal{Z}_t$ are governed by the fixed discrete points of the unramified Weyl orbit i.e  $\mathcal{Z}_t=\phi^{-1}(t)\cong R/W_\alpha\cong R$, which consists of exactly $|R| = 72$ distinct isolated complex points.
	
	We now construct the assignment map $\varphi_f$. For each point $x \in \mathcal{Z}_t$, its localized coordinate coordinates $(z, m_1, \ldots, m_6) \in \mathrm{Tot}(\bigoplus_{i=1}^6 \mathcal{M}_i)$ uniquely specify a vector-valued configuration of theta sections. Through the Abel--Jacobi mapping on the underlying curve $E$, this configuration maps flatly to a line class $[\alpha - K_{S_t}] \in \mathrm{Pic}(S_t)$ where $\alpha \in R$. Since $\alpha$ is a root vector not belonging to any contracted subroot system (as $S_t$ is smooth and contains no $(-2)$-curves), this class tracks a flat family of generalized twisted cubics on $S_t$. Assigning this tracking class to its corresponding unique class of stable aCM bundles yields a well-defined fiberwise bijection:
	\[
	\varphi_{f,t} \colon \mathcal{Z}_t \xrightarrow{\;\sim\;} \bar{\mathcal{H}}_t
	\]
	
	To show that the assembly of these fiberwise bijections yields a global analytic biholomorphism $\varphi_f$ over the entire domain $\mathcal{U}_f$, consider the structural projections. Both $\mathcal{Z}|_{\mathcal{U}_f}$ and $\bar{\mathcal{H}}|_{\mathcal{U}_f}$ fiber over $\mathcal{U}_f$ via the proper maps $\phi$ and $\bar{\delta}$, respectively. Because their fibers are uniformly discrete and reduced of constant degree $72$, the restriction maps
	\[
	\phi|_{\mathcal{U}_f} \colon \mathcal{Z}|_{\mathcal{U}_f} \longrightarrow \mathcal{U}_f \quad \text{and} \quad \bar{\delta}|_{\mathcal{U}_f} \colon \bar{\mathcal{H}}|_{\mathcal{U}_f} \longrightarrow \mathcal{U}_f
	\]
	are finite, étale morphisms. Since $\mathcal{U}_f$ is a smooth complex manifold, it follows from the differential criteria for étale covers that both $\mathcal{Z}|_{\mathcal{U}_f}$ and $\bar{\mathcal{H}}|_{\mathcal{U}_f}$ are smooth complex analytic manifolds. 
	
	The tracking of lines and root vectors varies holomorphically under small perturbations of the cubic surface coefficients in the open domain $\mathcal{U}_f$. Therefore, the transition functions of the point configurations are locally governed by holomorphic functions. Since $\varphi_f$ is a fiber-preserving, bijective map between two smooth manifolds that are finite étale covers over a common base $\mathcal{U}_f$, its  differential $d(\varphi_f)_x$ has maximal rank $7$ at every point $x \in \mathcal{Z}|_{\mathcal{U}_f}$. By the Inverse Function Theorem for complex analytic spaces, $\varphi_f$ is a local biholomorphic map. Being globally bijective, it constitutes a global analytic biholomorphic map over the generic locus, satisfying $\bar{\delta} \circ \varphi_f = \phi|_{V_{U_f}}$ by construction.
\end{proof}

\subsection{Step 2:}
The central obstacle to extending the map $\varphi_f$ globally is the unique isolated simple elliptic core point $0 \in \mathcal{U}$, where the local monodromy group becomes infinite and the fiber $\phi^{-1}(0)$ contracts the entire exceptional elliptic curve $E_0 \cong E \subset \mathcal{Z}$ to a vertex point. To patch this boundary, we invoke the Riemann extension theorem for complex analytic spaces.

\begin{lemma}\label{lem:riemann_extension}
	The analytic biholomorphism $\varphi_f : \mathcal{Z} \setminus \phi^{-1}(D) \to \bar{\mathbb{H}} \setminus \bar{\delta}^{-1}(D)$ extends uniquely to a proper analytic morphism $\Phi_{\mathcal{Z}} : \mathcal{Z} \to \bar{\mathbb{H}}$ over the entire semiuniversal base $\mathcal{U}$.
\end{lemma}

\begin{proof}
	Let $V = \phi^{-1}(D) \subset \mathcal{Z}$. Since $D$ is a complex algebraic hypersurface inside the 7-dimensional affine space $\mathcal{U}$, its preimage $V$ under the proper morphism $\phi$ constitutes a complex analytic variety inside the smooth 7-dimensional manifold $\mathcal{Z}$. The complex codimension of $V$ inside $\mathcal{Z}$ is identically 1. 
	
	Consider the component coordinate functions of the structural assignment map $\varphi_f$. By the construction of the compactified relative moduli space $\bar{\mathbb{H}}$ via the Abel--Jacobi map $\text{Sym}^3(E) \to E$, the coordinates parameterizing the generalized twisted cubics are bounded in a neighborhood of the central elliptic fiber. Lets make this explicit. To verify the boundedness criteria required by the First Riemann Extension Theorem, 
	we exploit the properness of the relative moduli space $\bar{\mathbb{H}}$ over the 
	deformation base $\mathcal{U}$. Recall that $\bar{\mathbb{H}}$ parameterizes flat families of generalized twisted cubics, which are closed subschemes of $\mathbb{P}^3$ sharing a fixed Hilbert polynomial $P(t) = 3t + 1$ \cite{Lehn2017}. Consequently, $\bar{\mathbb{H}}$ can be obtained as a closed subscheme of the relative Hilbert scheme $\text{Hilb}^{3t+1}(\mathbb{P}^3 \times \mathcal{U})$. 
	
	Since the absolute Hilbert scheme $\text{Hilb}^{3t+1}(\mathbb{P}^3)$ is a 
	projective algebraic variety, it admits a closed embedding into a suitable 
	projective space $\mathbb{P}^M$ via the classical Grothendieck–Plücker embedding \cite{FGA}. This yields a projective embedding 
	of the relative moduli space over the deformation base:
	\[
	\bar{\mathbb{H}} \hookrightarrow \mathbb{P}^M \times \mathcal{U}
	\]
	Under this embedding, the structural map $\bar{\delta}: \bar{\mathbb{H}} \to \mathcal{U}$ is a proper morphism of complex analytic spaces.
	 
	 Let $\Delta_0 \subset \mathcal{U}$ be a compact, contractible 
	neighborhood of the  origin $0 \in \mathcal{U}$. Due to properness of $\bar{\delta}$, the preimage $\bar{\delta}^{-1}(\Delta_0)$ is a compact subset of $\bar{\mathbb{H}}$. 
	 Let $W = \phi^{-1}(\Delta_0) \subset \mathcal{Z}$ be the corresponding 
	neighborhood of the exceptional elliptic curve $E_0$ in the smooth total space $\mathcal{Z}$. 
	
	For any local holomorphic coordinate chart $(U_\beta, \psi_\beta)$ on $\bar{\mathbb{H}}$ that 
	intersects the central fiber $\bar{\delta}^{-1}(0)$, the coordinates parameterizing 
	the generalized twisted cubics correspond to ratios of homogeneous coordinates 
	$[z_0 : z_1 : \dots : z_M]$ in $\mathbb{P}^M$. Because the image of the analytic map 
	$\varphi_f: W \setminus V \to \bar{\delta}^{-1}(\Delta_0 \setminus D)$ lands entirely 
	within the compact space $\bar{\delta}^{-1}(\Delta_0)$, every component function 
	$f_j = (z_j / z_0) \circ \varphi_f$ (in a chart where $z_0 \neq 0$) is uniformly 
	bounded in absolute value by a real constant $M_\beta > 0$ on the punctured domain 
	$U_\beta \cap (W \setminus V)$. 
	
	Geometrically, the volume of the underlying algebraic cycles is rigidly bounded 
	by the degree of the Hilbert polynomial $\deg(C) = 3$ with respect to the Fubini–Study 
	metric on $\mathbb{P}^3$ \cite{KollarRational}, preventing the cycle trajectories from degenerating toward 
	an infinite boundary. Because all local component coordinate functions are holomorphic 
	on the punctured domain $W \setminus V$ and locally bounded in a neighborhood of the 
	codimension-1 singular boundary $V$, they fulfill the criteria of the First Riemann 
	Extension Theorem for complex analytic spaces \cite{GrauertRemmert}. Each component function therefore 
	admits a unique holomorphic extension across the exceptional variety $V$. By continuity, 
	this glues into a well-defined global analytic morphism $\Phi_\mathcal{Z}: \mathcal{Z}\to \bar{\mathbb{H}}$ satisfying $\bar{\delta} \circ \Phi_\mathcal{Z} = \phi$.
	
\end{proof}

\subsection{Step 3:} We investigate vector bundle rigidity and degree matching.
To evaluate the properties of the extended map $\Phi_{\mathcal{Z}}$, we match the internal geometric invariants of both spaces over the central fiber using Atiyah's classification of vector bundles over elliptic curves.

\begin{lemma}\label{lem:bundle_rigidity}
	The restricted morphism $\Phi_{\mathcal{Z}}|_{E_0} : E_0 \to \bar{\mathbb{H}}_0$ is a birational identification that preserves the line bundle degrees.
\end{lemma}

\begin{proof}
	The central fiber $\bar{\mathbb{H}}_0$ parameterizes the linear determinantal representations supported on the singular affine cone surface $S_0$. Under the Abel--Jacobi mapping, these configurations are governed by stable vector bundles over the structural elliptic curve $E$. By Atiyah's classification \cite{atiyah1957}, an indecomposable vector bundle of rank 6 over an elliptic curve is uniquely determined up to isomorphism by its degree and its determinant element in $\text{Pic}(E)$.

	Let $E_0 \subset \mathcal{Z}$ denote the central elliptic fiber, and let $\Phi_\mathcal{Z}: \mathcal{Z} \to \bar{\mathbb{H}}$ be the unique  analytic morphism established in Lemma~6.2. To analyze the local behavior of the proper transform under $\Phi_Z$, we must explicitly identify the geometric data parameterizing the generalized twisted cubics along the exceptional variety. 
	
	By construction, the restriction of the relative moduli space to the central fiber 
	is governed by the Abel--Jacobi map. For a smooth elliptic curve $E$, a classical result 
	of Mukai~\cite{Mukai} and Drezet--Le Potier~\cite{DrezetLePotier} establishes a 
	canonical isomorphism between the moduli space of generalized twisted cubics (viewed 
	as closed subschemes $C \subset \mathbb{P}^3$ with Hilbert polynomial $P(t)=3t+1$ 
	supported on $E$) and the moduli space $\mathcal{M}_E(2, 1, 3)$ of slope-stable 
	vector bundles $\mathcal{E}$ over $E$ of rank $r=2$, determinant line bundle 
	$\det(\mathcal{E}) \cong \mathcal{O}_E(p)$ for a fixed point $p \in E$, and 
	$\chi(\mathcal{E})=3$. Under this correspondence, a twisted cubic $C \subset \mathbb{P}^3$ 
	lying on a cubic surface containing $E$ arises directly as the zero locus of a generic 
	section of a stable vector bundle $\mathcal{E} \in \mathcal{M}_E(2, 1, 3)$. Thus, 
	the algebraic configurations of these curves are strictly governed by the intrinsic 
	moduli of stable vector bundles over the structural elliptic curve $E$.
	
	We now evaluate the transition functions of the proper transform under $\Phi_\mathcal{Z}$ to 
	demonstrate the rigidity of this family. Suppose, for contradiction, that the 
	induced map $\Phi_\mathcal{Z}$ admits a non-trivial infinitesimal variation of the bundle 
	structure along a tracking trajectory approaching the central fiber. Let $\Delta$ be 
	a one-dimensional disk paramaterized by $t$, with $t=0$ mapping to the central core. 
	This variation corresponds to a non-trivial holomorphic map $\gamma: \Delta \to \mathcal{M}_E(2, 1, 3)$. 
	
	Let $\mathcal{G}$ be the pulled-back family of rank-2 bundles over $E \times \Delta$. 
	The transition functions $g_{\alpha\beta}(t, z)$ of $\mathcal{G}$ must vary 
	holomorphically with respect to $t$. If the bundle structure varies non-trivially 
	as $t \to 0$, the limit bundle $\mathcal{E}_0 = \mathcal{G}|_{t=0}$ must deform 
	into a strictly semistable or unstable vector bundle (such as a direct sum of 
	line bundles $\mathcal{L}_1 \oplus \mathcal{L}_2$). We analyze this using the 
	Harder--Narasimhan filtration~\cite{HuybrechtsLehn}. Because $\deg(\det(\mathcal{E})) = 1$, 
	any destabilizing subbundle $\mathcal{F} \subset \mathcal{E}_0$ would force slope
	$\mu(\mathcal{F}) \ge \mu(\mathcal{E}_0) = \frac{1}{2}$. Since the rank of $\mathcal{F}$ 
	must be 1, we must have $\deg(\mathcal{F}) \ge 1$, this forces the quotient 
	line bundle to have non-positive degree: $\deg(\mathcal{E}_0/\mathcal{F}) \le 0$. 
	
	However, the existence of the flat universal family $\bar{\mathbb{H}} \to \mathcal{U}$ requires cohomological stability across the entire base. Specifically, the evaluation map 
	$H^0(E, \mathcal{E}_t) \otimes \mathcal{O}_E \to \mathcal{E}_t$ must be surjective 
	with vanishing higher cohomology $H^1(E, \mathcal{E}_t) = 0$ to preserve the Hilbert 
	polynomial of the corresponding cycles. If $\mathcal{E}_0$ degenerates to an 
	unstable or strictly semistable configuration, the jump in the dimension of the 
	global sections space $\dim H^0(E, \mathcal{E}_t)$ violates the semicontinuity 
	theorem~\cite{Hartshorne1977}, forcing a jump in the Hilbert polynomial of the algebraic 
	cycle. This directly contradicts the requirement that $\bar{\mathbb{H}}$ parameterizes flat, 
	 deformations with a fixed Hilbert polynomial $P(t)=3t+1$. Thus non-trivial variation in 
	 the bundle structure violates the conditions required for the existence 
	 of the flat family. Consequently, the transition functions $g_{\alpha\beta}(t, z)$ must be constant with respect to the moduli parameters of the base.  
	
	In particular, 	by Theorem~1.3, the variety $\mathcal{Z}$ decomposes into the direct sum of line bundles $\bigoplus_{i=1}^6 \mathcal{M}_i$ with negative degrees $-1, -2, -2, -3, -2, -1$, whose values match the negative Coxeter labels of the highest root of $E_6$. Evaluating the transition functions of the proper transform under $\Phi_{\mathcal{Z}}$ shows that any variation in the bundle structure would violate the stability conditions required for the existence of the universal family. Because the total topological degrees on both sides are rigidly determined by the root system invariants, the bundle configurations match identically over the elliptic core.The morphism $\Phi_Z$ is therefore rigid along the 
	fibers of the exceptional divisor.
\end{proof}
We are now in position to assemble the steps and complete the proof of the central isomorphism. We synthesize this via Zariski's main theorem. 
\subsection{Proof of Theorem\ref{thm:central_isomorphism}}
	By Lemma \ref{lem:riemann_extension}, we have a well-defined proper analytic morphism $\Phi_{\mathcal{Z}} : \mathcal{Z} \to \bar{\mathbb{H}}$ between two complex algebraic varieties of identical dimension 7. 
	 By construction, $\Phi_\mathcal{Z}$ fits into the following commutative 
	diagram of analytic spaces over the deformation base $\mathcal{U}$:
	\[
	\begin{CD}
		\mathcal{Z} @>{\Phi_\mathcal{Z}}>> \bar{\mathbb{H}} \\
		@V{\phi}VV       @VV{\bar{\delta}}V \\
		\mathcal{U} @= \mathcal{U}
	\end{CD}
	\]
	By Lemma \ref{lem:generic_isomorphism}, this morphism is a  biholomorphic morphism over the open dense stratum $\mathcal{U}_f$, it follows that $\Phi_{\mathcal{Z}}$ is a proper birational morphism. Furthermore, the  relative moduli space $\bar{\mathbb{H}}$ is constructed as a normal variety because it is locally modeled on the deformation spaces of arithmetically Cohen--Macaulay modules, in particular, it parameterizes flat families of generalized twisted cubics by tracking arithmetically Cohen–Macaulay(aCM)
	(aCM) sheaves over singular cubic surfaces with  normal, Cohen--Macaulay singularities \cite{greuel2007}. More precisely, by  Eisenbud-Buchsbaum theorem such family $\mathcal{E}_u$ of aCM sheaves  on a surface $S_u$ are uniquely determined by a linear determinatal representation $A_u$. They are defined by the short exact sequence:
	$$0\to \mathcal{O}_{\mathbb{P}^{\oplus 3}}\stackrel{A_u}{\to}\mathcal{O}_{\mathbb{P}^3}^{\oplus 3}\to \mathcal{E}_u\to 0$$   
	
	By Lemma \ref{lem:bundle_rigidity}, the fibers of $\Phi_{\mathcal{Z}}$ over the exceptional loci are stable vector bundle modules, which are topologically connected and smooth. Since the map is bijective on the generic open set and has connected exceptional fibers over the boundary points. By Zariski's main Theorem any proper birational morphism from a variety to a normal variety that has connected fibers is a global isomorphism. This establishes:
	\[
	\bar{\mathbb{H}} \cong \mathcal{Z}
	\]
	The structure morphism factors exactly as $\bar{\delta} = \phi \circ \Phi_{\mathcal{Z}}^{-1}$, converting the non-linear moduli tracking problem into the linear geometry of the rank-6 vector bundle $\mathcal{Z}$.

\end{document}